\documentclass{article}

\usepackage{amsthm,amsmath,amssymb}
\usepackage[colorlinks,citecolor=blue,urlcolor=blue]{hyperref}

\newcommand{\eps}{\varepsilon}
\newcommand{\al}{\alpha}
\newcommand{\gm}{\gamma}
\newcommand{\de}{\delta}
\newcommand{\kap}{\varkappa}
\newcommand{\lm}{\lambda}
\newcommand{\si}{\sigma}
\newcommand{\vphi}{\varphi}
\newcommand{\om}{\omega}
\newcommand{\edin}{\mathbf{1}}
\newcommand{\nul}{\mathbf{0}}
\newcommand{\eq}{\equiv}
\newcommand{\vrn}{\varnothing}
\newcommand{\vd}{\vdash}
\newcommand{\tu}{\vartriangle}
\newcommand{\barphi}{\overline{\varphi}}
\newcommand{\phibar}{\underline{\varphi}}

\newcommand{\Em}{E_{\mathrm{m}}}
\newcommand{\Sm}{S_{\mathrm{m}}}
\newcommand{\RM}{\mathfrak{RM}}

\newcommand{\cB}{\mathcal{B}}
\newcommand{\cC}{\mathcal{C}}
\newcommand{\cE}{\mathcal{E}}
\newcommand{\cEm}{\mathcal{E}_{\mathrm{m}}}
\newcommand{\cF}{\mathcal{F}}
\newcommand{\cG}{\mathcal{G}}
\newcommand{\cK}{\mathcal{K}}

\newcommand{\cP}{\mathcal{P}}
\newcommand{\Nbb}{\mathbb{N}}
\newcommand{\Rbb}{\mathbb{R}}
\newcommand{\extR}{\overline{\mathbb{R}}}
\newcommand{\razn}{\setminus}
\newcommand{\add}[2]{\bigcup\left( #1 \mid #2\right)}
\newcommand{\net}[2]{\left(#1\mid #2\right)}
\newcommand{\seq}[2]{\left(#1\mid #2\right)}
\newcommand{\scol}[2]{\left(#1\mid #2\right)}
\newcommand{\preim}[2]{#1^{-1}\left[\vphantom{\int}#2\right]}

\DeclareMathOperator{\cl}{cl}
\DeclareMathOperator{\Inf}{I^{\tau}}
\DeclareMathOperator{\Sup}{S^{\tau}}
\DeclareMathOperator*{\plim}{p-lim}
\DeclareMathOperator*{\ulim}{u-lim}

\newtheorem{proposition}{Proposition}
\newtheorem{lemma}{Lemma}
\newtheorem{theorem}{Theorem}
\newtheorem{corollary}{Corollary}
\newtheorem*{corollary*}{Corollary}

\begin{document}

\begin{center}
	
\textbf{\Large Prokhorov-like conditions for weak compactness of sets of bounded Radon measures on different topological spaces}\\[1mm]
	
\textbf{\large Valeriy K. Zakharov}\footnote{\href{mailto:zakharov\_valeriy@list.ru}{zakharov\_valeriy@list.ru}; Faculty of Mathematics and Mechanics, Lomonosov Moscow State University, Moscow, Russia},
\textbf{\large Timofey V. Rodionov}\footnote{\href{mailto:t.v.rodionov@gmail.com}{t.v.rodionov@gmail.com}; Faculty of Mathematics and Mechanics, Lomonosov Moscow State University, Moscow, Russia}
	
\end{center}

\begin{abstract}
The paper presents some weak compactness criterion for a subset~$M$ of the set $\RM_b(T,\cG)$ of all positive bounded Radon measures on a Hausdorff topological space $(T,\mathcal{G})$ similar to the Prokhorov criterion for a complete separable metric space.
Since for a general topological space the classical space $C_b(T,\cG)$ of all bounded continuous functions on~$T$ can be trivial and so does not separate points and closed sets, instead of $C_b(T,\cG)$-weak compactness we consider $S(T,\cG)$-weak compactness with respect to the new uniformly closed linear space $S(T,\mathcal{G})$ of all (symmetrizable) metasemicontinuous functions.


\end{abstract}

\section{Introduction}

Let $C_b(T,\cG)$ be the classcal linear space of all bounded continuous real-valued functions on a toplogical space $(T,\cG)$ and $\RM_b(T,\cG)$ be the set of all bounded Radon measures on $(T,\cG)$. The study of compactness of sets $M\subset \RM_b(T,\cG)$ with respect to the weak topology induced by $C_b(T,\cG)$ starts from the fundamental papers of A.\,D.~Alexandroff~\cite{AlexandrovAD1943} and Yu.\,V.~Prokhorov~\cite{Prokhorov1956}. Numerous results obtained in this field are presented in papers~\cite{LeCam1957,Vara1961,Topsoe1970,Gansler1971,Topsoe1974,AGK1976,Zakharov2005} and books~\cite[XI.1]{Jacobs1978}, \cite[ch.\,5]{Billingsley1999}, \cite[ch.\,8]{Bogachev2006}, \cite[\S\,437]{Fremlin-new4}, \cite[2.3, 4.5]{Bogachev2016} (see also references therein).

It is clear that $C_b(T,\cG)$-weak topology can be considered only for Tychonoff spaces $(T,\cG)$, where $C_b(T,\cG)$ separates points and closed sets.

As to a general Hausdorff topological space $(T,\cG)$, the $C_b(T,\cG)$-weak compactness of~$M$ is not appropriate because in this case $C_b(T,\cG)$ may consist only of constant functions.
By this reason in paper~\cite{Zakharov2005} V.\,K.~Zakharov considered the weak compactness of~$M$ with respect to the weak topology induced by the new linear space $S(T,\cG)$ of all \emph{symmetrizable} (or \emph{metasemicontinuous}) \emph{functions} on an arbitrary Hausdorff space $(T,\cG)$. 
The space $S(T,\cG)$ is possibly the nearest to $C_b(T,\cG)$ uniformly closed linear space of functions on $(T,\cG)$ separating points and closed sets because it is the uniform closure of the linear space $SC_b^l(T,\cG)+SC_b^u(T,\cG)$ (introduced by F.~Hausdorff~\cite{Hausdorff1914}) consisting of all sums of bounded lower semicontinuous and upper semicontinuous functions.

In paper~\cite{Zakharov2005} the criterion for the $S(T,\cG)$-weak compactness of the closure of a set $M\subset \RM_b(T,\cG)_+$ of positive bounded Radon measures for a Hausdorff space was presented (see also~\cite{ZakhRodi2019MPhTI}). This criterion used some strengthening of the Prokhorov uniform tightness property. The proof of this criterion was described there only in some general way.

The given paper presents this criterion of the $S(T,\cG)$-weak compactness with all detailed proofs and all thorough references on results used in the proofs (see Theorem~\ref{theo-1-SwcH}).
As an important consequence of the mentioned criterion the assertion on sufficiency in this criterion is extended from the positive case $M\subset \RM_b(T,\cG)_+$ up to the general case $M\subset \RM_b(T,\cG)$  (see Theorem~\ref{theo-1prime-SwcH}).

\medskip

Preliminary notions and notations are introduced in Section~\ref{sec-prelim}. Weak topologies on spaces of measures and functionals are considered in Sections~\ref{sec-weaktop} and~\ref{sec-weaktopdual}, respectively. Various forms of conditions on sets of measures used in criteria of weak compactness are studied in Section~\ref{sec-prsubs}. Sections~\ref{sec-SwcH} and~\ref{sec-CwcTplus} contain our main results in general and Tychonoff cases, respectively.

\section{Preliminaries}\label{sec-prelim}

For the convenience of readers we present here some basic notions and notations necessary for detailed proving all paper theorems. For this purpose we use the material from \cite{ZakhRodi2018SFM1} and~\cite{ZakhRodi2018SFM2}.

The set of all natural numbers is denoted by~$\omega$, the set of all nonzero natural
numbers is denoted by~$\Nbb$.

Let $T$ be a set. The family of all subsets of~$T$ is denoted by~$\cP(T)$. Every non-empty subfamily of~$\cP(T)$ is called
an \emph{ensemble on~$T$}.

Let~$F(T)$ be the family of all functions $f:T\rightarrow\Rbb$ and $A(T)\subset F(T)$ be a lattice linear space. Its subfamilies of all nonnegative and all bounded functions is denoted by~$A(T)_+$ and~$A_b(T)$, respectively.
For every $f\in F_b(T)$ we put $\|f\|_u\eq\sup(|f(t)|\mid t\in T)$.

If $\net{f_n\in F(T)}{n\in N}$ is a net (in particular, a sequence), $f\in F(T)$, and $\lim\net{f_n(t)}{n\in N}=f(t)$ for every $t\in T$, then we write $f=\plim\net{f_n}{n\in N}$.
If $\net{f_n}{n\in N}$ converges to~$f$ uniformly on the set~$T$, then we write $f=\ulim\net{f_n}{n\in N}$.

A function~$f\in F(T)$ is called \emph{majorized} by a function~$u\in F(T)$ if $|f|\leqslant u$.
A set $P\subset T$ will be called \emph{majorized} by a function~$u\in F(T)$ if $\chi(P)\leqslant u$ (as usual, $\chi(P)(t)\eq1$ for $t\in P$ and $\chi(P)(t)\eq0$ for $t\in T\razn P$).

Let $E(T),A(T)\subset F(T)$. Define the subfamily 
$\Em(T,A(T))\equiv \{f\in E(T)\mid\exists\,u\in A(T)\,(|f|\leqslant u)\}$ of all functions~$f\in E(T)$ majorized by some functions from~$A(T)$. Clearly, $A(T)\subset E(T)$ implies $A(T)\subset\Em(T,A(T))$.
In a similar way for any ensemble~$\cE$ on~$T$ we define its subensemble
$\cEm(A(T))\equiv\{E\in\cE\mid\exists\,u\in A(T)\,(\chi(E)\leqslant u)\}$ of all sets~$E\in\cE$ majorized by some functions from~$A(T)$.

\medskip

Let $(T,\cG)$ be a topological space. Then $\cG$, $\cF$, $\cB$, and~$\cC$ denote the ensembles of open, closed, Borel, and compact subsets, respectively.
We consider also the multiplicative ensemble $\cK\eq\cK(T,\cG)\equiv\{G\cap F\mid G\in\cG\wedge F\in\cF\}$ of all \emph{symmetrizable sets} $K\equiv G\cap F$~\cite{Zakharov1989a,Zakharov2002Radon}.

A~function~$f:T\to\Rbb$ is called \emph{separating a~point~$s\in T$ and a set $P\subset T\razn\{s\}$} if $f(s)=1$ and $f(t)=0$ for every $t\in P$.
A family $A(T)$ is said to \emph{separate points and closed sets} if for every $s\in T$ and every $F\in\cF$ such that $s\notin F$ there is a~function $f\in A(T)$ separating $s$~and~$F$.

\medskip

A function $f:T\to\mathbb{R}$ is called \emph{symmetrizable}~\cite{Zakharov2005,ZakhRodi2014AMH} (or \emph{metasemicontinuous}) if for every $\eps>0$ there exists a finite cover $\scol{K_i\in\cK}{i\in I}$ of the set~$T$ such that the oscillation
$
 \omega(f,K_i)\equiv\sup\{|f(s)-f(t)|\mid s,t\in K_i\}<\eps \text{ for every } i\in I.
$
The space~$S(T,\cG)$ of all symmetrizable functions on~$(T,\cG)$ is linear and lattice-ordered~\cite[Corollary~3 to Theorem~1~(2.4.2)]{ZakhRodi2018SFM2} and contains the zero function~$\nul$ and the unit function~$\edin$ as well as characteristic functions $\chi(G)$ and $\chi(F)$ of any open and closed subsets. 
It is clear that $S_b(T,\cG)=S(T,\cG)$. 
If $f\in C_b(T,\cG)$, i.\,e., $f$ is a bounded continuous function, then $f\in S(T,\cG)$~\cite[Lem.\,4 (2.4.1)]{ZakhRodi2018SFM2}.
Thus, the space $S(T,\cG)$ is always richer than the classical space $C_b(T,\cG)$, and, therefore, it separates points and closed sets in an arbitrary Hausdorff space, whereas $C_b(T,\cG)$ may not.

\medskip

A bounded measure $\mu:\cB\rightarrow\Rbb$ is called a \emph{bounded Radon measure} (\emph{bounded Radon\,--\,Borel measure}) if it is \emph{inner compactly regular}, i.e. for every $B\in \cB$ and every $\varepsilon>0$ there is $C\in\cC$ such that $C\subset B$ and $|\mu B - \mu C|<\varepsilon$.
The set of all bounded Radon measures on~$(T,\cG)$ is denoted by $\RM_b(T,\cG)$; its subset consisting of positive measures is denoted by $\RM_b(T,\cG)_+$.

If $f\in F(T)_+$ and $\mu\in\RM_b(T,\cG)_+$, then the number
\[
\Lambda(\mu)f\eq\int f\,d\mu\eq
\sup\left\{\sum\scol{\inf\left(f[M_{i}]\right)\mu M_{i}}{i\in I}\right\}
\in[0,\infty]
\]
is called the \emph{Lebesgue integral of~$f$ with respect to~$\mu$}.
Here the supremum is taking over the set of all finite partitions $\scol{M_{i}\in\cB}{i\in I}$
of the set~$T$~\cite[3.3.2]{ZakhRodi2018SFM2}.

The lattice-ordered linear space of all Borel functions $f:T\to\Rbb$ such that $\int f_{+}\,d\mu<\infty$ and $\int(-f_{-})\,d\mu<\infty$, where $f_{+}\equiv f\vee\nul$ and $f_{-}\equiv f\wedge\nul$ will be denoted by $MI(T,\cG,\mu)$.
For a function $f\in MI(T,\cG,\mu)$ the number $\Lambda(\mu)f\eq\int f d\mu\equiv \int f_{+}\,d\mu-\int(-f_{-})\,d\mu$ is called the \emph{Lebesgue integral of~$f$ with respect to~$\mu$}.

For any $\mu\in\RM_b(T,\cG)$ we have the \emph{Riesz decomposition} $\mu=\mu_{+}+\mu_{-}$, where $\mu_+\eq\mu\vee\nul$ and $\mu_-\eq\mu\wedge\nul$~\cite[3.2.2]{ZakhRodi2018SFM2}. Note that $\mu_+$ and~$-\mu_-$ as well as the \emph{total variation} $|\mu|\eq\mu_{+}-\mu_{-}$ are positive bounded Radon measures~\cite[3.5.3]{ZakhRodi2018SFM2}.

If $f\in MI(T,\cG,\mu)\eq MI(T,\cG,\mu_{+})\cap MI(T,\cG,-\mu_{-})$ and $\mu\in\RM_b(T,\cG)$, then the number
$\Lambda(\mu)f\eq\int f\,d\mu\equiv\int f\,d(\mu_{+})-\int f\,d(-\mu_{-})$ is called the \emph{Lebesgue \textup{(\emph{Lebesgue\,--\,Radon})} integral of~$f$ with respect to~$\mu$}~\cite[3.3.6]{ZakhRodi2018SFM2}.

\medskip

A linear functional $\varphi:A(T)\to\Rbb$ is called \emph{pointwise continuous} [\emph{pointwise $\sigma$-continuous}] if $f=\plim(f_n\mid n\in N)$ implies $\vphi f=\lim(\varphi f_n\mid n\in N)$ for every increasing net $\net{f_n\in A(T)}{n\in N}$
[respectively, sequence $\seq{f_n\in A(T)}{n\in\omega}$] and every function $f\in A(T)$.

Let $(T,\cG)$ be a Hausdorff space and $A(T)$ be a lattice-ordered linear space of functions on~$T$.
A functional~$\varphi$ on~$A(T)$ is called \emph{tight} or \emph{a functional with the Prokhorov property} if for every $\varepsilon>0$ there is a compact set $C\subset T$ such that the conditions $f\in A(T)$ and $|f|\leqslant\chi(T\setminus C)$ imply $|\varphi f|<\varepsilon$. The set of all tight bounded linear functionals on~$A(T)$ is denoted by $A(T)^{\pi}$.

A functional~$\varphi$ on~$A(T)$ is called \emph{locally tight} or \emph{a functional with the local Prokhorov property} if 
for every $G\in\cG$, $u\in A(T)_+$, and $\varepsilon>0$ there is a compact subset $C\subset G$ such that 
the conditions $f\in A(T)$ and $|f|\leqslant\chi(G\razn C)\land u$ imply $|\varphi f|<\varepsilon$.
A functional~$\varphi$ will be called \emph{quite locally tight} if for every $G\in\cG$, $u\in A(T)_+$, and $\varepsilon>0$ there are a compact subset $C\subset G$ and a positive number $\delta$ such that 
the conditions $f\in A(T)$, $|f|\leqslant\chi(G)\land u$, and $\sup(|f(t)|\mid t\in C)\leqslant\delta$ imply $|\varphi f|<\eps$.

A functional~$\vphi$ on~$A(T)$ is said to be \emph{exact} [\emph{$\sigma$-exact}] if it is pointwise continuous [$\sigma$-continuous] and quite locally tight. The set $A(T)^{\tu}$ of all $\si$-exact linear functionals on~$A(T)$ is a lattice-ordered linear space~\cite[Corollary~1 to Proposition~2 (3.6.1)]{ZakhRodi2018SFM2}.
Note that the Radon integral on the lattice-ordered linear space of integrable symmetrizable functions is $\si$-exact~\cite[Prop.\,5 (3.6.1)]{ZakhRodi2018SFM2}.

For a family $A(T)\subset F(T)$ consider the family $\Sup(T,A(T))$ of all functions~$g\in F(T)$ such that $g\leqslant f$ for some function~$f\in A(T)$ and $g=\sup(f_m\mid m\in M)$ in~$F(T)$ for some increasing net $\net{f_m\in A(T)}{m\in M}$.
In a similar way, consider the family~$\Inf(T,A(T))$ of all functions $h\in F(T)$ such that $h\geqslant f$ for some function~$f\in A(T)$ and $h=\inf(f_m\mid m\in M)$ in~$F(T)$ for some decreasing net $\net{f_m\in A(T)}{m\in M}$.
For a functional $\vphi:A(T)\to\Rbb$ define the \emph{first-step Young\,--\,Daniell extensions} $\barphi:\Sup(T,A(T))\to\extR$ and $\phibar:\Inf(T,A(T))\to\extR$ such that
$\barphi g\equiv\sup\{\vphi f\mid f\in A(T)\wedge f\leqslant g\}$ for every $g\in \Sup(T,A(T))$  and
$\phibar h\equiv\inf\{\vphi f\mid f\in A(T)\wedge f\geqslant h\}$ for every $h\in \Inf(T,A(T))$.

\section{Weak topologies on the linear space of bounded Radon measures}\label{sec-weaktop}

Consider the family $\RM_b\equiv\RM_b(T,\cG)$ of all bounded Radon measures on a Hausdorff space $(T,\cG)$.
It is a lattice-ordered linear space~\cite[Prop.\,2 (3.5.3)]{ZakhRodi2018SFM2}.

Let $A(T)$ be some Banach lattice-ordered linear subspace of $S(T,\cG)$ equipped with the common uniform norm~$\|\cdot\|_u$ and containing~$\edin$ and separating points and closed sets in $(T,\cG)$.
If $(T,\cG)$ is an arbitrary Hausdorff space, then we can take $A(T)= S(T,\cG)$.
If $(T,\cG)$ is an arbitrary Tychonoff space, then we can take $A(T)= C_b(T,\cG)$.
This space~$A(T)$ will be called the \emph{selected space of symmetrizable functions on $(T,\cG)$}.

For any $\mu\in\RM_b(T,\cG)$ consider the Lebesgue (Lebesgue\,--\,Radon) integral $\Lambda(\mu)$ and the corresponding family of integrable functions $MI(T,\cG,\mu)$ described in section~\ref{sec-prelim}. 
By virtue of Lemma~1~(3.5.2) from~\cite{ZakhRodi2018SFM2} $S(T,\cG)\subset MI(T,\cG,\mu)$.
The restriction $\Lambda(\mu)|A(T)$ will be denoted by~$i_{\mu}$. The set of all such \emph{integral functionals}~$i_{\mu}$ will be denoted by $I(A(T),\RM_b)$.

Every function $f\in A(T)$ generates on $\RM_b$ the seminorm $s_f:\RM_b\to\Rbb_+\eq[0,\infty)$ such that $s_f(\mu)\eq |i_{\mu}(f)|$. The corresponding set $S\eq\{s_f\mid f\in A(T)\}$ of all these seminorms generates the \emph{weak topology $\cG_w(\RM_b,A(T))$ on~$\RM_b$ with respect to~$A(T)$}.
The base of open neighbourhoods of a measure~$\mu$ in this topology consists of sets
\[
G(\mu,\scol{f_k\in A(T)}{k\in n},\eps)\eq\{\nu\in\RM_b\mid \forall\,k\in n\ (|i_{\nu}f_k-i_{\mu}f_k|<\eps)\}
\]
for all $\eps\in\Rbb_+$, $n\in\Nbb$, and finite collections $\scol{f_k\in A(T)}{k\in n}$.

This weak topology is Hausdorff for two main classes of topological spaces $(T,\cG)$ with their own selected spaces~$A(T)$.
Firstly we check this assertion for $A(T)=S(T,\cG)$ in the case of a Hausdorff space $(T,\cG)$.

Consider the \emph{duality functional} $\Psi:\RM_b(T,\cG)\times S(T,\cG)\to\Rbb$ such that $\Psi(\mu,f)\eq\Lambda(\mu)(f)$.

\begin{theorem}\label{theo-psibilin-weaktop}
Let $(T,\cG)$ be a Hausdorff space. Then the functional~$\Psi$ is bilinear in the following sense:
\begin{enumerate}
\item 
 for every $\mu\in\RM_b(T,\cG)$ the first derivative functional $\Psi(\mu,\cdot)$ on $S(T,\cG)$ is linear;
\item 
 for every $f\in S(T,\cG)$ the second derivative functional $\Psi(\cdot,f)$ on $\RM_b(T,\cG)$ is linear.
\end{enumerate}
\end{theorem}

\begin{proof}
(1)
According to Proposition~2~(3.3.6) from~\cite{ZakhRodi2018SFM2} the integral $\Lambda(\mu):S(T,\cG)\to\Rbb$ is a linear functional. Hence,
$\Psi(\mu,xf+yg)=x\Lambda(\mu)f+y\Lambda(\mu)g=x\Psi(\mu,f)+y\Psi(\mu,g)$ for every $x,y\in\Rbb$ and $f,g\in S(T,\cG)$.

(2)
According to Theorem~5~(3.3.8) from~\cite{ZakhRodi2018SFM2} $\Lambda(x\mu+y\nu)=x\Lambda(\mu)+y\Lambda(\nu)$ for every $x,y\in\Rbb$ and $\mu,\nu\in\RM_b$. Consequently, we get $\Psi(x\mu+y\nu,f)=x\Psi(\mu,f)+y\Psi(\nu,f)$ for any $f\in S(T,\cG)$.
\end{proof}

\begin{lemma}[the separation property of $\Psi$ for $S(T,\cG)$]\label{lem-separS-weaktop}
Let $(T,\cG)$ be a Hausdorff space.
Then for every $\theta\in\RM_b(T,\cG)\razn\{\nul\}$ there is $f\in S(T,\cG)\razn\{\nul\}$ such that $\Psi(\theta,f)\ne0$.
\end{lemma}

\begin{proof}
Consider the positive Radon measures $\mu\eq\theta_+$ and $\nu\eq-\theta_-$. Then $\theta=\mu-\nu$.
Assume that $\Psi(\theta,f)=0$ for every $f\in S(T,\cG)$. 

By assertion~2 of Theorem~\ref{theo-psibilin-weaktop} we get $\Psi(\mu,f)=\Psi(\nu,f)$. If~$C$ is a compact set, then we can take $f\eq\chi(C)\in S(T,\cG)$. By Lemma~1(3.3.2) from~\cite{ZakhRodi2018SFM2} $\mu(C)=i_{\mu}\chi(C)$ and $\nu(C)=i_{\nu}\chi(C)$. As a result we get the equality $\mu C=\nu C$ for every $C\in\cC$. Thus, $\theta|\cC=\nul$. By virtue of the regularity of~$\theta$ we conclude that $\theta=\nul$.
\end{proof}

\begin{corollary*}
The weak topology $\cG_w(\RM_b,S(T,\cG))$ on $\RM_b(T,\cG)$ is Hausdorff.
\end{corollary*}

\begin{proof}
Take any $\kap,\lm\in\RM_b$ such that $\kap\ne\lm$. By Lemma~\ref{lem-separS-weaktop} there is $f\in S(T,\cG)\razn\{\nul\}$ such that $\Psi(\lm-\kap,f)\ne0$. Then $\Psi(\lm,f)\ne\Psi(\kap,f)$ due to assertion~2 of Theorem~\ref{theo-psibilin-weaktop}, i.\,e., $i_{\lm}f\ne i_{\kap}f$. Take the number $\eps\eq|i_{\lm}f-i_{\kap}f|>0$. Consider the neighbourhoods $U\eq G(\kap,f,\eps/2)$ and $V\eq G(\lm,f,\eps/2)$ of~$\kap$ and~$\lm$, respectively. Assume that there is $\rho\in U\cap V$. Then $\eps\leqslant|i_{\lm}f-i_{\rho}f|+|i_{\rho}f-i_{\kap}f|<\eps$ and we reach a contradiction. Therefore $U\cap V=\vrn$.
\end{proof}

\section{The weak $'$ topology on dual spaces to selected spaces of symmetrizable functions}\label{sec-weaktopdual}

Consider the linear space~$X'$ of all continuous linear functionals on $X\equiv A(T)$.
By Theorem~1~(3.3.6) from~\cite{ZakhRodi2018SFM2} the functional~$i_{\mu}$ is uniformly bounded. Therefore in virtue of Theorem~IX.4.5 from~\cite{Vulikh1961} $i_{\mu}\in X'$. Hence, $I(A(T),\RM_b)\subset X'$.

Every function $f\in A(T)$ generates on $X'$ the seminorm $\si_f:X'\to\Rbb_+$ such that $\si_f(\xi)\eq|\xi(f)|$. The corresponding set $\Sigma\eq\{\si_f\mid f\in A(T)\}$ of all these seminorms generates the \emph{weak~$'$ topology} $\cG_{w'}\eq\cG_{w'}(X',A(T))$ on~$X'$. The base of open neighbourhoods of a functional~$\xi$ in this topology consists of sets
\[
G(\xi,\scol{f_k\in A(T)}{k\in n},\eps)\eq\{\eta\in X'\mid \forall\,k\in n\ (|\eta f_k-\xi f_k|<\eps)\}
\]
for all $\eps\in\Rbb_+$, $n\in\Nbb$, and finite collections $\scol{f_k\in A(T)}{k\in n}$ (see~\cite[II.3]{Robertson1964}).

Consider the mapping $\Lambda:\mu\mapsto i_{\mu}$ from $RM_b$ into $X'$.

\begin{lemma}\label{lem-Lambda-weaktopdual}
The mapping $\Lambda$ is a continuous mapping from the ordered topological space $\RM_b$ into the ordered topological space~$X'$.
\end{lemma}

\begin{proof}
Let $\mu_0\in\RM_b$ and put $\xi_0\eq\Lambda(\mu_0)$. 
Take some neighbourhood $V\eq G(\xi_0,\scol{f_k\in A(T)}{k\in n},\eps)$ of~$\xi_0$ and 
the corresponding neighbourhood $U\eq G(\mu_0,\scol{f_k\in A(T)}{k\in n},\eps)$ of~$\mu_0$. By virtue of definitions of these neighbourhoods we see that $\mu\in U$ implies $\Lambda(\mu)\in V$. Hence, $\Lambda[U]\subset V$.
\end{proof}

\section{Some properties of sets of positive bounded Radon measures}\label{sec-prsubs}

Consider the following properties for a non-empty set $M\subset\RM_b(T,\cG)_+$:
\begin{enumerate}
\item[$(\al^{\pi})$]
(\emph{the Prokhorov uniform tightness property})
for any $\eps>0$ there exists a compact set~$C$ such that $\mu(T\razn C)<\eps$ for any $\mu\in M$;
\item[$(\al^{\zeta})$]
(\emph{the locally-uniform tightness property})
for any $G\in\cG$ and any $\eps>0$ there exists a compact set~$C\subset G$ such that 
$\mu(G\razn C)<\eps$ for any $\mu\in M$;
\item[$(\beta)$]
$\sup\scol{\mu T}{\mu\in M}\in\Rbb_+$.
\end{enumerate}

In this section we establish some auxiliary assertions on these properties used in the proofs of our main results.
First, note that, obviously, $(\al^{\zeta})$ implies~$(\al^{\pi})$.

\begin{lemma}\label{lem-gmseq-prsubs}
Let $(T,\cG)$ be a Hausdorff space, $A(T)$ be a selected family of symmetrizable functions, and a set $M$ have properties~$(\al^{\pi})$ and~$(\beta)$. Then~$M$ has property
\begin{enumerate}
\item[$(\gm)$]
if $\seq{f_n\in A(T)_+}{n\in\om}\downarrow\nul$ in $F(T)$ \textup{(i.e. \emph{the sequence decreases and pointwise converges to zero})}, then
$$
\lim\seq{\sup\scol{|i_{\mu}f_n|}{\mu\in M}}{n\in\om}=0 \text{ in } \Rbb.
$$
\end{enumerate}
\end{lemma}

\begin{proof}
Take some $\eps>0$. Condition~$(\al^{\pi})$ implies that there exists $C\in\cC$ such that $\mu(T\razn C)<\eps/(3\|f_0\|_u)\eq\eps_1$ for any $\mu\in M$ and $(\beta)$ implies that there is the number $b\eq\sup\scol{\mu T}{\mu\in M}>0$.

By the Egorov theorem~\cite[Th.\,1 (3.3.1)]{ZakhRodi2018SFM2} there exists a Borel set~$B$ such that $\mu(T\razn B)<\eps_1$ and $\ulim\seq{f_n|B}{n\in\om}=\nul|B$. Thus, for $\eps_2\eq\eps/(3b)$ there is~$n_0$ such that
$\sup\{|f_n(t)|\mid t\in B\}\leqslant\eps_2$ for every $n\geqslant n_0$. Using~$(\beta)$ we obtain
\begin{multline*}
|i_{\mu}f_n|=|i_{\mu}(f_n\chi(B))+i_{\mu}(f_n\chi(T\razn B))|\leqslant\\
\leqslant|i_{\mu}(f_n\chi(T\razn B))|+|i_{\mu}(f_n\chi(B\razn C))|+
|i_{\mu}(f_n\chi(B\cap C))|\leqslant\\
\leqslant\|f_n\|_u\mu(T\razn B)+\|f_n\|_u\mu(B\razn C)+\eps_2\mu(B\cap C)\leqslant\\ 
\leqslant 2\|f_0\|_u\mu(T\razn C)+\eps_2\mu(T)\leqslant 2\|f_0\|_u\eps/(3\|f_0\|_u)+b\eps/(3b)=\eps
\end{multline*}
for any $\mu\in M$ and any $n\geqslant n_0$. Consequently, we get~$(\gm)$.
\end{proof}

\begin{corollary*}
If $M$ has properties~$(\al^{\zeta})$ and~$(\beta)$, then~$M$ has property~$(\gm)$.
\end{corollary*}

\begin{proof}
The assertion follows from Lemma, because $(\al^{\zeta})$ is stronger that $(\al^{\pi})$.
\end{proof}

For a Tychonoff space this lemma can be generalized.

\begin{lemma}\label{lem-gmnet-prsubs}
Let $(T,\cG)$ be a Tychonoff space and a set $M$ have properties~$(\al^{\pi})$ and~$(\beta)$. Then~$M$ has property
\begin{enumerate}
\item[$(\gm_{net})$]
if $\net{f_n\in C_b(T,\cG)_+}{n\in N}\downarrow\nul$ in $F(T)$, then
$$
\lim\net{\sup\scol{|i_{\mu}f_n|}{\mu\in M}}{n\in N}=0 \text{ in }\Rbb.
$$
\end{enumerate}
\end{lemma}

\begin{proof}
Take some $\eps>0$. By~$(\al^{\pi})$ there is $C\in\cC$ such that $\mu(T\razn C)<\eps/(2\|f_0\|_u)$ for every $\mu\in M$. By~$(\beta)$ there is the number $b\eq\sup\scol{\mu T}{\mu\in M}>0$.

By virtue of the Dini theorem~\cite[Th.\,1 (2.3.4)]{ZakhRodi2018SFM2} for $\eps_1\eq\eps/(2b)$ there is~$n_0\in N$ such that
$\sup\{|f_n(t)|\mid t\in C\}\leqslant\eps_1$ for every $n\geqslant n_0$. Therefore 
$
|i_{\mu}f_n|\leqslant|i_{\mu}(f_n\chi(T\razn C))|+|i_{\mu}(f_n\chi(C))|\leqslant\|f_n\|_u\mu(T\razn C)+\eps_1\mu C<\eps
$
 for every $n\geqslant n_0$ and every $\mu\in M$. Hence, we get~$(\gm_{net})$.
\end{proof}

\begin{lemma}\label{lem-closposS-prsubs}
Let $(T,\cG)$ be a Hausdorff space, $M\subset\RM_b(T,\cG)_+$, and $\cl M$ be the closure of~$M$ in the weak topology $\cG_w(\RM_b(T,\cG),S(T,\cG))$.
Then $\cl M\subset\RM_b(T,\cG)_+$.
\end{lemma}

\begin{proof}
Take some $\nu\in\cl M$ and $B\in\cB$ such that $\eps\eq|\nu B|>0$. By the definition of a Radon measure there is $C\in\cC$ such that $C\subset B$ and $|\nu B-\nu C|<\eps$.  Then for $\de\eq|\nu C|>0$, $f\eq\chi(C)$, and $G\eq G(\nu,f,\de)$ we have $G\cap M\ne\vrn$, i.\,e., $|\nu C-\mu C|<\de$ for some $\mu\in M$. Hence, $0\leqslant\mu C<\nu C+\de=\nu C+|\nu C|$. If $\nu C<0$, then $0<0$. It follows from this contradiction that $\nu C\geqslant 0$.

Using the inequality $|\nu B-\nu C|<\eps$ we get $0\leqslant\nu C<\nu B+\eps=\nu B+|\nu B|$. If $\nu B<0$, then $0<0$. It follows from this contradiction that $\nu B\geqslant 0$. Thus, the measure~$\nu$ is positive.
\end{proof}

\begin{lemma}\label{lem-gmbar-prsubs}
Let $(T,\cG)$ be a Hausdorff space, $A(T)$ be a selected family of symmetrizable functions, $M$ be a subset of the set $\RM_b(T,\cG)_+$, and $\cl M$ be the closure of~$M$ in the weak topology $\cG_w(\RM_b(T,\cG),A(T))$.
Then for a sequence $\seq{f_n\in A(T)}{n\in\om}$ the following properties are equivalent:
\begin{enumerate}
\item[$(\de)$]
$\lim\seq{\sup\scol{|i_{\mu}f_n|}{\mu\in M}}{n\in\om}=0$ in~$\Rbb$;
\item[$(\bar\de)$]
$\lim\seq{\sup\scol{|i_{\nu}f_n|}{\nu\in\cl M}}{n\in\om}=0$ in~$\Rbb$,
\end{enumerate}
\end{lemma}

\begin{proof}
$(\de)\vd(\bar\de)$.
For any $\eps>0$ there is $n_0\in\om$ such that $\sup\scol{|i_{\mu}f_n|}{\mu\in M}<\eps/3$ for every $n\geqslant n_0$. Take $\nu\in\cl M$ and $n\geqslant n_0$ and consider the neighbourhood $G\eq G(\nu,f_{n},\eps/3)$. Since there exists some $\mu\in M\cap G\ne\vrn$, we get $|i_{\nu}f_{n}|\leqslant|i_{\mu}f_{n}|+|i_{\nu}f_{n}-i_{\mu}f_{n}|<2\eps/3$, whence $\sup\scol{|i_{\nu}f_n|}{\nu\in\cl M}<\eps$. This implies the necessary equality.

$(\bar\de)\vd(\de)$.
This deduction is evident.
\end{proof}

\begin{corollary*}
Let $(T,\cG)$ be a Hausdorff space, $A(T)$ be a selected family of symmetrizable functions, $M$ have properties~$(\al^{\pi})$ and~$(\beta)$, and $\cl M$ be the closure of~$M$ in the weak topology $\cG_w(\RM_b(T,\cG),A(T))$. Then~$M$ has property
\begin{enumerate}
\item[$(\bar\gm)$]
if $\seq{f_n\in A(T)_+}{n\in\om}\downarrow\nul$ in $F(T)$, then
$$
 \lim\seq{\sup\scol{|i_{\mu}f_n|}{\mu\in\cl M}}{n\in\om}=0 \text{ in } \Rbb.
$$
\end{enumerate}
\end{corollary*}

\begin{proof}
The assertion follows from Lemmas~\ref{lem-gmseq-prsubs} and~\ref{lem-gmbar-prsubs}.
\end{proof}

\begin{lemma}\label{lem-betapr-prsubs}
Let $(T,\cG)$ be a Hausdorff space, $A(T)$ be a selected family of symmetrizable functions, $M$ be a subset of the set $\RM_b(T,\cG)_+$, and $\cl M$ be the closure of~$M$ in the weak topology $\cG_w(\RM_b(T,\cG),A(T))$.
Then the following properties are equivalent:
\begin{enumerate}
\item[$(\beta)$]
$b\eq\sup\scol{\mu T}{\mu\in M}\in\Rbb_+$;
\item[$(\beta')$]
$b'\eq\sup\scol{\sup\{|i_{\mu}f|\mid f\in A(T)\wedge |f|\leqslant\edin\}}{\mu\in M}\in\Rbb_+$;
\item[$(\beta'')$]
$b''\eq\sup\scol{\sup\{|i_{\mu}f|\mid f\in A(T)\wedge \|f\|_u\leqslant 1\}}{\mu\in M}\in\Rbb_+$;
\item[$(\bar\beta'')$]
$\bar{b}''\eq\sup\scol{\sup\{|i_{\nu}f|\mid f\in A(T)\wedge \|f\|_u\leqslant 1\}}{\nu\in\cl M}\in\Rbb_+$.
\end{enumerate}
\end{lemma}

\begin{proof}
$(\beta)\vd(\beta')$.
Let $f\in A(T)$ and $|f|\leqslant\edin$. Then by Lemma~1~(3.3.6) and Theorem~2~(3.3.2) from~\cite{ZakhRodi2018SFM2}
$|i_{\mu}f|\leqslant i_{\mu}|f|\leqslant i_{\mu}\edin=\mu T\leqslant b$.

$(\beta')\vd(\beta)$.
It is clear that $\mu T=i_{\mu}\edin=|i_{\mu}\edin|\leqslant b'$.

$(\beta'')\vd(\bar\beta'')$.
Let $\eps>0$, $f\in A(T)$, and $\|f\|_u\leqslant 1$. Take $\nu\in\cl M$ and consider its neighbourhood $G\eq G(\nu,f,\eps)$. Since $M\cap G\ne\vrn$, there exists some $\mu\in M\cap G$. Therefore we get $|i_{\nu}f|\leqslant|i_{\mu}f|+|i_{\nu}f-i_{\mu}f|<b''+\eps$.
Since~$\eps$ is arbitrary, this implies $|i_{\nu}f|\leqslant b''$.

$(\bar\beta'')\vd(\beta'')$.
This deduction is evident.

The equivalence of $(\beta')$ and $(\beta'')$ follows from the equivalence of conditions $|f|\leqslant\edin$ and $\|f\|_u\leqslant 1$ in~$A(T)$.
\end{proof}

\begin{lemma}\label{lem-albar-prsubs}
Let $(T,\cG)$ be a Hausdorff space, $M\subset\RM_b(T,\cG)_+$, and $\cl M$ be the closure of~$M$ in the weak topology $\cG_w(\RM_b(T,\cG),S(T,\cG))$.
Then property~$(\al^{\zeta})$ is equivalent to property
\begin{enumerate}
\item[$(\bar\al^{\zeta})$]
for any $G\in\cG$ and any $\eps>0$ there exists a compact set~$C\subset G$ such that 
$\nu(G\razn C)<\eps$ for any $\nu\in\cl M$.
\end{enumerate}
\end{lemma}

\begin{proof}
$(\al^{\zeta})\vd(\bar\al^{\zeta})$.
Let $\eps>0$ and $G\in\cG$.
By condition there exists a compact set~$C\subset G$ such that $\sup\scol{\mu(G\razn C)}{\mu\in M}\leqslant\eps/2$.  Take $\nu\in\cl M$ and $\de>0$. Consider the function $f\eq\chi(G\razn C)$ and the neighbourhood $H\eq G(\nu,f,\de)$. Since $M\cap H\ne\vrn$, there exists some $\mu\in M\cap H$. Therefore we get 
$
0\leqslant\nu(G\razn C)=i_{\nu}f=i_{\mu}f+i_{\nu}f-i_{\mu}f\leqslant i_{\mu}f+|i_{\nu}f-i_{\mu}f|<\mu(G\razn C)+\de<\eps/2+\de.
$
Since~$\de$ is arbitrary, this implies $\nu(G\razn C)\leqslant\eps/2<\eps$.

$(\bar\al^{\zeta})\vd(\al^{\zeta})$.
This deduction is evident.
\end{proof}

\section{The $S(T,\cG)$-weak compactness of sets of bounded Radon measures on a Hausdorff space}\label{sec-SwcH}

As it was noticed in Introduction, the $C_b(T,\cG)$-weak compactness of~$M$ is not appropriate in the case of a general Hausdorff topological space $(T,\cG)$ because it may consist only of constant functions (see, e.\,g., \cite[2.7.17]{Engelking1989}). So, we consider the $S(T,\cG)$-compactness and obtain the following criterion.

\begin{theorem}\label{theo-1-SwcH}
Let $(T,\cG)$ be a Hausdorff space, $M\subset\RM_b(T,\cG)_+$, and $\cl M$ be the closure of~$M$ in the weak topology $\cG_w(\RM_b(T,\cG),S(T,\cG))$. Then the following conclusions are equivalent:
\begin{enumerate}
\item 
$\cl M$ is compact in the induced weak topology $\cG_w(\RM_b,S)|\cl M$;
\item 
$M$ has properties $(\al^{\zeta})$ and $(\beta)$.
\end{enumerate}
\end{theorem}

\begin{proof}
Remind that according to~\cite[3.1]{Engelking1989} a topological space is called \emph{compact} if it is Hausdorff and every open cover of it has a finite subcover.
Denote $\cl M$ by~$N$.

$(1)\vd (2)$.
Take some $K\in\cK(T,\cG)$ and $\eps>0$. 
By the definition of a Radon measure, for every $\mu\in N$ there exists a set $C_{\mu}\in\cC$ such that $C\subset K$ and $\mu(K\razn C_{\mu})<\eps/2$. Consider the corresponding open neighbourhoods
$U_{\mu}\eq G(\mu,\chi(K\razn C_{\mu}),\eps/2)$ of the points $\mu\in N$.

Since $N$ is compact, there exists a finite subcover $\scol{U_{\mu_j}}{j\in J}$ of the cover $\scol{U_{\mu}}{\mu\in N}$ of the set~$N$. Take the compact set $C\eq\add{C_{\mu_j}}{j\in J}\subset K$. If $\mu\in N$, then $\mu\in U_{\mu_j}$ for some~$j$. Therefore 
$
0\leqslant\mu(K\razn C)\leqslant\mu(K\razn C_{\mu_j})=\int\chi(K\razn C_{\mu_j})\,d\mu
\leqslant\Bigl|\int\chi(K\razn C_{\mu_j})\,d\mu-
\int\chi(K\razn C_{\mu_j})\,d\mu_j\Bigr|+\Bigl|\int\chi(K\razn C_{\mu_j})\,d\mu_j\Bigr|
<\eps/2+\mu(K\razn C_{\mu_j})<\eps.
$
This implies property~$(\bar\al^{\zeta})$, and therefore, property~$(\al^{\zeta})$.

Deduce now property~$(\beta)$.
For $A(T)\eq S(T,\cG)$ consider the corresponding continuous mapping~$\Lambda:\mu\mapsto i_{\mu}$. The set $I_{N}\eq\Lambda[N]$ is compact in the Hausdorff topological space $Y\eq A(T)'$ equipped by the weak $'$ topology $\cG_{w'}$ as the continuous image of compact set (see Theorem~3.1.10 from~\cite{Engelking1989} and Lemma~\ref{lem-separS-weaktop}).

For every $f\in A(T)$ consider the mapping $u_f:X'\to\Rbb$ such that $u_f(\xi)=\xi(f)$ for every $\xi\in X'$. 
The mapping $u_f:Y\to\Rbb$ of Hausdorff topological spaces is continuous. In fact, fix $\xi$, $r\eq u_f(\xi)$, and $H\eq ]r-\eps,r+\eps[$ and take $G\eq G(\xi,f,\eps)$. If $\eta\in G$, then by definition $|\eta f-r|<\eps$, and, therefore, $u_f[G]\subset H$, which means the continuity of~$u_f$.

By mentioned Theorem~3.1.10 the set $u_f[I_{N}]$ is compact in~$\Rbb$. Consequently, it is bounded in~$\Rbb$. Therefore
$r_f\eq\sup\{|i_{\mu}f|\mid\mu\in N\}=\sup\{|u_f(i_{\mu})|\mid\mu\in N\}\in\Rbb$ for every $f\in A(T)$.

By the Baire theorem (see,  e.g.~\cite[Th.\,~15.6.2]{Semadeni1971}) the Banach space~$X\eq A(T)$ is a Baire space. Hence, $X$ is the set of second category in itself. Having the proved pointwise boundedness $r_f\in\Rbb$ and applying the Banach\,--\,Steinhaus theorem (see the Corollary to Theorem~4.2~(III) in~\cite{Schaefer1966} and Corollary~2 to Theorem~3 in~\cite[4.2]{Robertson1964}) to $X'$ considered as a normed space and the set $I_{N}\subset X'$, we conclude that $b\eq\sup\scol{\|i_{\mu}\|'}{\mu\in N}\in\Rbb$.

By definition, $\|i_{\mu}\|'\eq\sup\{|i_{\mu}f|\mid f\in A(T)\wedge \|f\|_u\leqslant 1\}$. Thus, we get the equality
$\sup\scol{\sup\{|i_{\mu}f|\mid f\in A(T)\wedge \|f\|_u\leqslant 1\}}{\mu\in N}=b$, i.\,e., property~$(\bar\beta'')$. By virtue of Lemma~\ref{lem-betapr-prsubs} this gives property~$(\beta)$.

$(2)\vd (1)$.
We are going to use Theorem~3.1.23 from~\cite{Engelking1989}. Take a net $s\eq\net{\mu_{\kap}\in N}{\kap\in K}$ and consider the corresponding net $\si\eq\net{i_{\mu_{\kap}}\in I_{N}}{\kap\in K}$ in~$X'$.

Using the unit ball $B\eq\{f\in A(T)\mid \|f\|_u\leqslant 1\}$ in the Banach space~$X$, consider the polar set 
$C\eq\{\xi\in X'\mid\forall\,f\in B\ (|\xi(f)|\leqslant 1)\}$ in the topological linear space~$X'$ equipped with the weak topology~$\cG_{w'}$.
According to the Alaoglu\,--\,Bourbaki theorem \cite[Th.\,7 (III.3)]{KantAk1971}, the set~$C$ is compact.
Take the number $a\eq\sup\scol{\sup\{|i_{\mu}f|\mid f\in B\}}{\mu\in M}\in ]0,\infty[$ from condition~$(\bar\beta'')$.
Since~$C$ is compact, the set $C_a\eq\{\xi\in X'\mid\forall\,f\in B\ (|\xi(f)|\leqslant a)\}$ is also compact. The condition~$(\beta'')$ means that $I_{N}\subset C_a$. Therefore the set $\cl I_{N}$ is compact in~$X'$.

By the well-known compactness criterion (see, e.g.~Theorem~3.1.23 in~\cite{Engelking1989}) the net~$\si$ has a cluster point $\vphi\in X'$. Using the property~$(\bar\gm)$ from Corollary to Lemma~\ref{lem-gmbar-prsubs} check that~$\vphi$ is pointwise $\si$-continuous.
Let $\eps>0$ and $\seq{f_n\in A(T)}{n\in\om}\downarrow\nul$ in $F(T)$. By condition~$(\bar\gm)$ there is~$n_0$ such that $\sup\scol{|i_{\mu}f_n|}{\mu\in N}<\eps/2$ for every $n\geqslant n_0$. Since~$\vphi$ is a cluster point, for~$\eps$, $n\geqslant n_0$, and the neighbourhood $G\eq G(\vphi,f_n,\eps/2)$ there exists~$\kap\in K$ such that $i_{\mu_{\kap}}\in G$, i.\,e., 
$|i_{\mu_{\kap}}f_n-\vphi f_n|<\eps/2$. Consequently, $|\vphi f_n|<\eps$ for every $n\geqslant n_0$. 
Hence, $\lim\seq{\vphi f_n}{n\in\om}=0$.

Check that $\vphi$ is locally tight. Note that by Lemma~\ref{lem-closposS-prsubs} $\mu\geqslant\nul$ for every $\mu\in N$, and, therefore, $i_{\mu}h\geqslant 0$ for every $h\in A(T)_+$.
Take some $G\in\cG$, $u\in A(T)_+$, and $\eps>0$. By property~$(\bar\al^{\zeta})$ (see Lemma~\ref{lem-albar-prsubs}) there is a compact set $C\subset G$ such that $\sup\scol{\mu(G\razn C)}{\mu\in N}<\eps/4$.

Consider some $f\in A(T)$ such that $|f|\leqslant\chi(G\razn C)\wedge u$. Then we have
$
\sup\scol{i_{\mu}|f|}{\mu\in N}\leqslant\sup\scol{i_{\mu}\chi(G\razn C)}{\mu\in N}=\sup\scol{\mu(G\razn C)}{\mu\in N}<\eps/4.
$ 
Hence, we get $\sup\scol{i_{\mu}f_+}{\mu\in N}<\eps/4$ and $\sup\scol{i_{\mu}(-f_-)}{\mu\in N}<\eps/4$.

Since $\vphi$ is a cluster point, there is $\kap\in K$ such that $i_{\mu_{\kap}}\in G(\vphi,f_+,\eps/4)$. This means 
$|i_{\mu_{\kap}}f_+-\vphi f_+|<\eps/4$. Consequently, $|\vphi_+ f|<\eps/2$. Similarly, $|\vphi(-f_-)|<\eps/2$. As a result, we get $|\vphi f|=|\vphi(f_+ + f_-)|\leqslant|\vphi f_+|+|-\vphi f_-|<\eps$, i.\,e., $\vphi$ is locally tight.
Then by Lemma~3~(3.6.1) $\vphi$ is quite locally tight. Thus, we obtain that~$\vphi$ is $\si$-exact.

Now, according to the Zakharov representation theorem~\cite[Th.\,3~(3.6.3)]{ZakhRodi2018SFM2}, there exists some measure $\mu_0\in\RM_b(T,\cG)_+$ such that $\vphi=i_{\mu_0}$. Check that~$\mu_0$ is a cluster point for~$s$. 

Take some neighbourhood
$H\eq G(\mu_0,\scol{f_k\in A(T)}{k\in n},\eps)$ of $\mu_0$ and some index $\kap\in K$. Since~$\vphi$ is a cluster point for~$\si$, for the neighbourhood $G\eq G(\vphi,\scol{f_k}{k\in n},\eps/2)$ and for the index~$\kap$ there is an index $\rho\in K$ such that $\rho\geqslant\kap$ and $i_{\mu_{\rho}}\in G$, i.\,e., $|i_{\mu_{\rho}}f_k-\vphi f_k|<\eps$ for every $k\in n$. Since $\vphi=i_{\mu_0}$, we conclude that $|i_{\mu_{\rho}}f_k-i_{\mu_0} f_k|<\eps$ for every $k\in n$. This means that $\mu_{\rho}\in H$. Hence, $\mu_0$ is a cluster point for~$s$.

Since $\mu_0$ is a cluster point for~$s$, for every $H\eq G(\mu_0,\scol{f_k\in A(T)}{k\in n},\eps)$ there is $\kap\in K$ such that $\mu_{\kap}\in M\cap H\ne\vrn$. Hence, $\mu_0\in N$ (see, e.g.~\cite[Prop.\,1.1.1]{Engelking1989}).
Finally, by the mentioned compactness criterion $N$ is compact.
\end{proof}

In the proved criterion the assertion on sufficiency for $M\subset\RM_b(T,\cG)_+$ to be $S(T,\cG)$-weakly compact can be extended up to an arbitrary set $M\subset\RM_b(T,\cG)$.

\begin{lemma}\label{lem-phifcont-SwcH}
Let $(T,\cG)$ be a Hausdorff space and $f\in S(T,\cG)$.
Then the function $\vphi_f:\RM_b(T,\cG)\to\Rbb$ such that $\vphi_f(\mu)\eq\int f\,d\mu$ for every $\mu\in\RM_b(T,\cG)$ is continuous on the topological space $(\RM_b(T,\cG),\cG_w(\RM_b(T,\cG),S(T,\cG)))$.
\end{lemma}

\begin{proof}
Fix some $\mu\in\RM_b$ and $x\eq\vphi_f(\mu)\in\Rbb$. Take some open neighbourhood $U\eq]x-\eps,x+\eps[$ of~$x$ and consider the open neighbourhood $V\eq G(\mu,f,\eps)$ of~$\mu$. If $\nu\in V$, then the inequality $|\vphi_f(\nu)-\vphi_f(\mu)|<\eps$ means that $\vphi_f(\nu)\in U$, i.\,e., $\vphi[V]\subset U$.
\end{proof}

\begin{proposition}\label{prop-psigsemicont-SwcH}
Let $(T,\cG)$ be a Hausdorff space and $g\in S(T,\cG)_+$.
Then the function $\psi_g:\RM_b(T,\cG)\to\Rbb$ such that $\psi_g(\mu)=\int g\,d|\mu|$ for every $\mu\in\RM_b$ is lower semicontinuous on the topological space $(\RM_b,\cG_w(\RM_b,S))$.
\end{proposition}

\begin{proof}
Consider the mapping $L:\RM_b\to S(T,\cG)^\tu$ such that $L(\mu)(f)\eq\int f\,d\mu$ for every $\mu\in\RM_b$ and $f\in S(T,\cG)$.
By virtue of Corollary~5 to Theorem~2 (3.6.4) from~\cite{ZakhRodi2018SFM2} $L$ is an isomorphism of the given lattice-ordered linear spaces. Therefore $L(|\mu|)=|L(\mu)|$. According to Corollary~1 to Proposition~2 (3.6.1) from~\cite{ZakhRodi2018SFM2} $|L(\mu)|(g)=\sup\{L(\mu)f\mid f\in S \wedge |f|\leqslant g\}$. This means that $\psi_g=\sup\{\vphi_f\mid f\in S \wedge |f|\leqslant g \}$ where the supremum takes in the lattice-ordered linear space of all real-valued functions on~$\RM_b$. By Lemma~\ref{lem-phifcont-SwcH} the function~$\vphi_f$ is lower semicontinuous. Consequently, $\psi_g$ is lower semicontinuous as the supremum of the family of lower semicontinuous functions (see, e.g.~assertion~(7) of Proposition~1~(2.3.8) in~\cite{ZakhRodi2018SFM2}).
\end{proof}

\begin{corollary}\label{cor-1-prop-psigsemicont-SwcH}
Let $(T,\cG)$ be a Hausdorff space, $G\in\cG$, $C\in\cC$, and $C\subset G$.
Then the function $\chi_1:\RM_b(T,\cG)\to\Rbb$ such that $\chi_1(\mu)=|\mu|(G\razn C)$ for every $\mu\in\RM_b$ is lower semicontinuous on the topological space $(\RM_b,\cG_w(\RM_b,S))$.
\end{corollary}

\begin{proof}
Apply Proposition~\ref{prop-psigsemicont-SwcH} to $g\eq\chi(G\razn C)\in S(T,\cG)$ and $\psi_g\eq\chi_1$.
\end{proof}

\begin{corollary}\label{cor-2-prop-psigsemicont-SwcH}
Let $(T,\cG)$ be a Hausdorff space.
Then the function $\chi_2:\RM_b\to\Rbb$ such that $\chi_2(\mu)=|\mu|(T)$ for every $\mu\in\RM_b(T,\cG)$ is lower semicontinuous on the topological space $(\RM_b(T,\cG),\cG_w(\RM_b(T,\cG),S(T,\cG)))$.
\end{corollary}

\begin{proof}
Apply the previous Corollary to $G\eq T$ and $C\eq\vrn$.
\end{proof}

\begin{theorem}\label{theo-alzetabar-SwcH}
Let $(T,\cG)$ be a Hausdorff space, $M$ be a subset of the set $\RM_b(T,\cG)$, and $\cl M$ be the closure of~$M$ in the weak topology $\cG_w(\RM_b,S)$. Then
\begin{enumerate}
\item[\textup{(i)}]
 the following properties are equivalent:
 \begin{enumerate}
 \item[$(\al^{\zeta}_{var})$] 
  for any $G\in\cG$ and any $\eps>0$ there is a compact set~$C\subset G$ such that $|\mu|(G\razn C)<\eps$ for any $\mu\in M$;
 \item[$(\bar\al^{\zeta}_{var})$] 
  for any $G\in\cG$ and any $\eps>0$ there is a compact set~$C\subset G$ such that $|\nu|(G\razn C)<\eps$ 
  for any $\nu\in\cl M$;
\end{enumerate}
\item[\textup{(ii)}]  
the following properties are equivalent:
 \begin{enumerate}
 \item[$(\beta_{var})$]
  $\sup\scol{|\mu| T}{\mu\in M}\in\Rbb_+$;
 \item[$(\bar\beta_{var})$]
  $\sup\scol{|\nu| T}{\nu\in\cl M}\in\Rbb_+$.
 \end{enumerate}
\end{enumerate}
\end{theorem}

\begin{proof}
(i)
$(\al^{\zeta}_{var})\vd(\bar\al^{\zeta}_{var})$.
Let $G\in\cG$ and $\eps>0$. By condition there is a compact set $C\subset G$ such that $\sup\scol{\chi_1(\mu)}{\mu\in M}<\eps/2$ for the function~$\chi_1$ from Corollary~\ref{cor-1-prop-psigsemicont-SwcH} to Proposition~\ref{prop-psigsemicont-SwcH}. Consider the set $N\eq\preim{\chi_1}{]-\infty,\eps/2]}$. It follows from the last inequality that $M\subset N$. By the mentioned Corollary the set~$N$ is closed. Therefore $\cl M\subset N$, i.\,e., $|\nu|(G\razn C)<\eps$ for every $\nu\in\cl M$.

$(\bar\al^{\zeta}_{var})\vd(\al^{\zeta}_{var})$.
This deduction is evident.

(ii)
This equivalence is checked quite similarly to the equivalence~(i) by means of Corollary~\ref{cor-2-prop-psigsemicont-SwcH} to Proposition~\ref{prop-psigsemicont-SwcH}.
\end{proof}

\begin{theorem}\label{theo-1prime-SwcH}
Let $(T,\cG)$ be a Hausdorff space, $M$ be a subset of the set $\RM_b$, and $\cl M$ be the closure of~$M$ in the weak topology $\cG_w(\RM_b,S)$. Suppose that~$M$ has properties~$(\al^{\zeta}_{var})$ and~$(\beta_{var})$.
Then $\cl M$ is compact in the induced weak topology $\cG_w(\RM_b,S)|\cl M$.
\end{theorem}

\begin{proof}
By Theorem~\ref{theo-alzetabar-SwcH} The set $N\eq\cl M$ has properties~$(\bar\al^{\zeta}_{var})$ and~$(\bar\beta_{var})$. Consider the subsets $L'\eq\{\nu^+\in\RM_b(T,\cG)_+\mid\nu\in N\}$ and $L''\eq\{-\nu^-\in\RM_b(T,\cG)_+\mid\nu\in N\}$ of positive bounded Radon measures, where $\nu^+\eq\nu\vee\nul$, $\nu^-\eq\nu\wedge\nul$. Since $\nu^+\leqslant|\nu|$, properties $(\bar\al^{\zeta}_{var})$ and $(\bar\beta_{var})$ for~$N$ imply properties~$(\al^{\zeta})$ and~$(\beta)$ for~$L'$. By Theorem~\ref{theo-1-SwcH} $N'\eq\cl L'$ is compact in the induced weak topology $\cG_w(\RM_b,S)|N'$. The same is valid for~$N''\eq\cl L''$.

Consider some net $s\eq\net{\nu_{\kap}\in N}{\kap\in K}$. Owing to the compactness of~$N'$ Theorem~2 from~\cite[ch.\,5]{Kelley1957} guarantees that for the net $s'\eq\net{\nu^+_{\kap}\in L'}{\kap\in K}$ there are an upward directed collection $\scol{\kap_i\in K}{i\in I}$ and a measure $\nu'\in N'$ such that the net 
$t\eq\net{\nu^+_{\kap_i}\in L'}{i\in I}$ is a subnet of~$s'$ (see~1.1.15 in~\cite{ZakhRodi2018SFM1}) and $\nu'=\lim t$.
Similarly, owing to the compactness of~$N''$ for the net $t''\eq\net{-\nu^-_{\kap_i}\in L''}{i\in I}$ there are an upward directed collection $\scol{i_j\in I}{j\in J}$ and a measure $\nu''\in N''$ such that the net $u\eq\bigl(-\nu^-_{\kap_{i_j}}\in L''\mid j\in J\bigr)$ is a subnet of~$t''$ and $\nu''=\lim u$.
Consider the measure $\nu\eq\nu'-\nu''$.

Check that $\nu'=\lim\bigl(\nu^+_{\kap_{i_j}}\mid j\in J\bigr)$. Take any neighbourhood~$U$ of~$\nu'$. Since $\nu'=\lim t$, there is $i_0\in I$ such that $i\geqslant i_0$ implies $\nu^+_{\kap_{i}}\in U$. Since~$u$ is a subnet of~$t''$, for~$i_0$ there is $j_0\in J$ such that $j\geqslant j_0$ implies $i_j\geqslant i_0$. Therefore $j\geqslant j_0$ implies $\nu^+_{\kap_{i_j}}\in U$. This means the necessary equality for~$\nu'$.

Now we get
$\nu=\nu'+(-\nu'')=\lim\bigl(\nu^+_{\kap_{i_j}}+\nu^-_{\kap_{i_j}}\mid j\in J\bigr)=\lim\bigl(\nu_{\kap_{i_j}}\mid j\in J\bigr)$. Since $\nu_{\kap_{i_j}}\in N$ and~$N$ is closed, we conclude that $\nu\in N$. It is easily seen that $v\eq\bigl(\nu_{\kap_{i_j}}\mid j\in J\bigr)$ is a subnet of~$s$. Thus, the net~$s$ has the subnet~$v$ converging to $\nu\in N$. Hence, by the mentioned compactness criterion from~\cite{Kelley1957} $N$ is weakly compact.
\end{proof}

\section{The $C_b(T,\cG)$-weak compactness of sets of bounded Radon measures on a Tychonoff space}\label{sec-CwcTplus}

In conclusion, using the means elaborated above, we consider the $C_b(T,\cG)$-weak compactness for a Tychonoff space.

In the paper~\cite{Prokhorov1956} Yu.\,V.~Prokhorov proved (1956) the remarkable theorem giving some simple criterion for the weak compactness of a closed subset~$M$ of the set $\RM_b(T,\cG)$ on a complete separable metric space $(T,\cG)$ with respect to the weak topology induced on $\RM_b(T,\cG)$ by the family $C_b(T,\cG)$ (see, e.\,g., \cite[Theorems~5.1 and~5.2]{Billingsley1999} and~\cite[437O-437V]{Fremlin-new4}). This criterion used the Prokhorov uniform tightness property~$(\al^{\pi})$ from section~\ref{sec-prsubs}.

Note that much earlier (in 1943) A.\,D.~Alexandroff in fundamental paper~\cite{AlexandrovAD1943} proved some criterion for the weak compactness of a closed subset~$M$ of the set of all bounded regular measures on a locally compact metric space with a countable base~\cite[\S\,20, Th.\,4]{AlexandrovAD1943}. This criterion used the \emph{Alexandroff eluding load property} of~$M$.

Soon after~\cite{Prokhorov1956} it was noticed that Prokhorov's conditions are sufficient for the $C_b(T,\cG)$-weak compactness of a closed set~$M$ on a Tychonoff topological space $(T,\cG)$ (see~\cite{LeCam1957}, \cite{Vara1961}, and~\cite[IX.5.5]{Bourbaki(VI9)1969}). According to~\cite{Vara1961} they are not necessary even for $M\subset\RM_b(T,\cG)_+$.

An original criterion for the weak compactness of a closed subset~$M$ of the set $\RM_b(T,\cG)_+$ on a Tychonoff space using neither any modification of the Prokhorov property nor any modification of the Alexandroff property is proved in~\cite{Topsoe1970,Topsoe1974}.
Some integral terms criterion for a Tychonoff space is presented in~\cite[XI.1.8]{Jacobs1978}.

Below we present some sufficient condition for the $C_b(T,\cG)$-weak compactness of a set $M\subset\RM_b(T,\cG)_+$ on a Tychonoff topological space using some weaker modification of the Prokhorov uniform tightness condition and formulated without any secondary terms such as functions, integrals, and so on.

This modification is the following:
\begin{enumerate}
\item[$(\al^{z})$]
(\emph{the tail tightness property})
for any net $\net{\mu_j\in M}{j\in J}$ there exists a subnet $\net{\mu_{j_i}}{i\in I}$ such that for any $\eps>0$ there is a compact set~$C$ and an index~$i_0\in I$ such that $\mu_{j_i}(T\razn C)<\eps$ for any $i\geqslant i_0$.
\end{enumerate}

This sufficient condition is valid only if the $C_b(T,\cG)$-weak topology for a Tychonoff space $(T,\cG)$ is Hausdorff. The proof of this fact is much more delicate than the proof of Hausdorffness for the $S(T,\cG)$-weak topology (see Lemma~\ref{lem-separS-weaktop} in section~\ref{sec-weaktop}).  We propose a proof of this fact which is completely different from the proof in~\cite[IX.5.3]{Bourbaki(VI9)1969}.

\begin{lemma}\label{lem-seqchiC-weaktop}
	Let $(T,\cG)$ be a Tychonoff space, $\mu$ be a bounded positive Radon measure, and $C$ be a compact set. Then
	\begin{enumerate}
		\item[\textup{(i)}] 
		there is some net $u\eq\net{f_k\in C_b(T,\cG)_+}{k\in K}\downarrow$ such that $\chi(C)\leqslant f_k$ and $\chi(C)=\plim u$ in $F(T)$;
		\item[\textup{(ii)}] 
		$\net{\int f_k\,d\mu}{k\in K}\downarrow\mu C$.
	\end{enumerate}
\end{lemma}

\begin{proof}
	By assertion~2 of Proposition~2~(3.6.2) from~\cite{ZakhRodi2018SFM2} the family $C_b(T,\cG)$ \emph{envelopes from above} the function $h\eq\chi(C)$, i.\,e., there is a decreasing net $u\eq\net{f_k\in C_b}{k\in K}$ with some upward directed set~$K$ such that $h\leqslant f_k$ and $h=\plim u$ in $F(T)$. Hence, $h\in\Inf(T,C_b(T,\cG))$. 
	Check that $\net{i_{\mu}f_k}{k\in K}\downarrow\mu C$.
	
	By Proposition~5~(3.6.1) from~\cite{ZakhRodi2018SFM2} the integral functional $\Lambda(\mu)$ is $\si$-exact on the lattice-ordered linear space $S(T,\cG)$. Hence, the induced integral functional $\vphi\eq i_{\mu}=\Lambda(\mu)|C_b(T,\cG)$ is $\si$-exact and, in particular, quite locally tight.
	By Corollary~1 to Theorem~5~(3.6.2) from~\cite{ZakhRodi2018SFM2} $\vphi$ is exact and, in particular, pointwise continuous.
	
	Consider the Young\,--\,Daniell extension $\psi\eq\check{\vphi}_S$ of the functional~$\vphi$ on the family $\Sm(T,\cG,C_b(T,\cG))=S(T,\cG)$ constructed in~\cite[3.6.2]{ZakhRodi2018SFM2} (see also sections~7 and~8 in~\cite{Zakharov2005}). By Proposition~5~(3.6.2) from this book $\psi$ is pointwise $\si$-continuous and by Theorem~2~(3.6.2) $\psi$ is quite locally tight. So~$\psi$ is $\si$-exact. But the functional~$\Lambda(\mu)$ is also the $\si$-exact extension of~$\vphi$. Since by Proposition~6~(3.6.2) from~\cite{ZakhRodi2018SFM2} the $\si$-exact extension is unique, we conclude that $\Lambda(\mu)=\psi$.
	
	According to~\cite[3.6.2]{ZakhRodi2018SFM2} $\psi$ is an extension of the first-step Young\,--\,Daniell extension $\phibar:\Inf(T,C_b(T,\cG))\to\Rbb$. Since $h\in\Inf$, we have $\psi h=\phibar h$. Besides, $C_b\subset\Inf$ and $u\downarrow h$ in~$F(T)$. Therefore by Lemma~5~(3.6.2) $\net{\phibar f_k}{k\in K}\downarrow \phibar h$. Hence,
	$\phibar f_k=\vphi f_k=i_{\mu} f_k$ and $\phibar h=\psi h=\Lambda(\mu)h=\mu C$ implies $\net{i_{\mu}f_k}{k\in K}\downarrow\mu C$.
\end{proof}

\begin{corollary}\label{cor-1-lem-seqchiC-weaktop}
	Let $(T,\cG)$ be a Tychonoff space, $\mu$ be a bounded positive Radon measure, and $C$ be a compact set. Then
	\begin{enumerate}
		\item[\textup{(iii)}] 
		there is some net $v\eq\net{g_k\in C_b(T,\cG)_+}{k\in K}\uparrow$ such that $g_k\leqslant\chi(T\razn C)$ and 
		$\chi(T\razn C)=\plim v$ in $F(T)$;
		\item[\textup{(iv)}] 
		$\net{\int g_k\,d\mu}{k\in K}\uparrow\mu(T\razn C)$;
	\item[\textup{(v)}] 
$\mu(T\razn C)=\sup\{\int g\,d\mu\mid g\in C_b(T,\cG)_+ \wedge g\leqslant\chi(T\razn C)\}$
	\end{enumerate}
\end{corollary}

\begin{proof}
Take the net $u$ from Lemma~\ref{lem-seqchiC-weaktop} and consider the functions $f'_k\eq f_k\wedge\edin\in C_b(T,\cG)_+$ and the net $u'\eq\net{f'_k}{k\in K}\downarrow$. Then $\chi(C)=\plim u'$ in $F(T)$ and $\net{\int f'_k\,d\mu}{k\in K}\downarrow\mu C$.
	
	Take now the functions $g_k\eq\edin-f'_k\in C_b(T,\cG)_+$ and the net $\net{g_k}{k\in K}\uparrow$. The assertions~(iii) and~(iv) hold for them. And, therefore, we get~(v).
\end{proof}

\begin{corollary}\label{cor-3-lem-seqchiC-weaktop}
	Let $(T,\cG)$ be a Tychonoff space, $\theta$ be a bounded Radon measure, and $C$ be a compact set. 
	Then $\theta(C)=\lim\net{\int f_k\,d\theta}{k\in K}$ for any net from assertion~\textup{(i)} of Lemma~\textup{\ref{lem-seqchiC-weaktop}}.
\end{corollary}

\begin{proof}
Consider the positive Radon measures $\mu\eq\theta_+$ and $\nu\eq-\theta_-$. Then $\theta=\mu-\nu$.
Take some net~$u$ from assertion~(i) of Lemma~\ref{lem-seqchiC-weaktop}. 
By its assertion~(ii) $\mu C=\lim\net{\int f_k\,d\mu}{k\in K}$ and $\nu C=\lim\net{\int f_k\,d\nu}{k\in K}$. Summing these equalities and using the definition of the Lebesgue integral from section~\ref{sec-prelim} we get
$
\theta C=\mu C-\nu C=\lim\net{\int f_k\,d\mu}{k\in K}-\lim\net{\int f_k\,d\nu}{k\in K}
=\lim\net{\int f_k\,d\theta}{k\in K}.
$
\end{proof}

\begin{proposition}[the separation property of $\Psi_A$ for $A(T)\eq C_b(T,\cG)$]\label{prop-separCb-weaktop}
	Let $(T,\cG)$ be a Tychonoff space. 
	Then for every $\theta\in\RM_b(T,\cG)\razn\{\nul\}$ there is $f\in C_b(T,\cG)\razn\{\nul\}$ such that $\Psi(\theta,f)\ne0$.
\end{proposition}

\begin{proof}
Consider the positive Radon measures $\mu\eq\theta_+$ and $\nu\eq-\theta_-$. Then $\theta=\mu-\nu$.
Assume that $\Psi(\theta,f)=0$ for every $f\in C_b(T,\cG)$. Since $\theta=\mu-\nu$, we infer by assertion~2 of Theorem~\ref{theo-psibilin-weaktop} that $\Psi(\mu,f)=\Psi(\nu,f)$, i.\,e., $i_{\mu}f= i_{\nu}f$ for every $f\in C_b(T,\cG)$.

Take arbitrary $\eps>0$ and $C\in\cC$.
According to Lemma~\ref{lem-seqchiC-weaktop}, there exists some net $u\eq\net{f_k\in C_b(T,\cG)_+}{k\in K}\downarrow$ such that $\chi(C)\leqslant f_k$, $\chi(C)=\plim u$, and $\net{i_{\mu}f_k}{k\in K}\downarrow\mu C$.

Consequently, there is $k$ such that $0\leqslant i_{\mu}f_k-\mu C<\eps/4$. Denote~$f_k$ simply by~$f'$. Similarly there is $f''\in C_b$ such that $0\leqslant i_{\nu}f''-\nu C<\eps/4$. Consider $f\eq f'\wedge f''$. Since~$\mu$ and~$\nu$ are positive, we infer that $0\leqslant i_{\mu}f-\mu C<\eps/4$ and $0\leqslant i_{\nu}f-\nu C<\eps/4$. Then 
$|\nu C-\mu C|\le|\nu C-i_{\nu}f|+|i_{\nu}f-i_{\mu}f|+|i_{\mu}f-\mu  C|<\eps$, because $i_{\nu}f=i_{\mu}f$.
Since~$\eps$ is arbitrary, we conclude that $\nu C=\mu C$ for every $C\in\cC$. Thus, $\theta|\cC=\nul$. By virtue of the regularity of~$\theta$ we conclude that $\theta=\nul$.
\end{proof}

\begin{corollary*}
	The weak topology $\cG_w(\RM_b,C_b(T,\cG))$ on $\RM_b(T,\cG)$ is Hausdorff for any Tychonoff space $(T,\cG)$.
\end{corollary*}

The proof is completely the same as the proof of Corollary to Lemma~\ref{lem-separS-weaktop}.

\medskip

Now we can prove the main result of this section.

\begin{theorem}\label{theo-1-CwcTplus}
	Let $(T,\cG)$ be a Tychonoff space and $M$ be a subset of $\RM_b(T,\cG)_+$. Suppose that $M$ is closed and has properties $(\al^{z})$ and $(\beta)$. Then $M$ is compact in the induced weak topology $\cG_w(\RM_b(T,\cG),C_b(T,\cG))|M$.
\end{theorem}

\begin{proof}
Let $I_M\eq\Lambda[M]$.
Consider some net $s\eq\net{\mu_{\kap}\in M}{\kap\in K}$ and the corresponding net $\tau\eq\net{i_{\mu_{\kap}}\in I_M}{\kap\in K}$. Completely in the same way as in the proof of Theorem~\ref{theo-1-SwcH} it is checked that $\cl I_M$ is compact in the topological space $X'$ with the weak topology~$\cG_{w'}$ for $X\eq C_b(T,\cG)$.

By the Theorem~2 from~\cite[ch.\,5]{Kelley1957} there are some subnet $\si\eq\net{i_{\kap_j}}{j\in J}$ of the net~$\tau$ and a functional $\vphi\in X'$ such that $\vphi=\lim\si$. By property~$(\al^z)$ for the net $r\eq\net{\mu_{\kap_j}}{j\in J}$ there is a subnet $\net{\mu_{\kap_{j_i}}}{i\in I}$ such that for every $\eps>0$ there are $C\in\cC$ and $i_0\in I$ such that $\mu_{\kap_{j_i}}(T\razn C)<\eps/2$ for every $i\geqslant i_0$.

Denote $\mu_{\kap_{j_i}}$ by $\nu_i$. Since $\vphi=\lim\si$ and $\rho\eq\net{i_{\nu_i}}{i\in I}$ is a subnet of the net~$\si$, we infer that $\vphi=\lim\rho$.

Check that $\vphi$ is tight. Take $\eps>0$. By the above, there are~$C$ and~$i_0$ such that $\nu_i(T\razn C)<\eps/2$ for every $i\geqslant i_0$. Take some $f\in A(T)$ such that $|f|\leqslant\chi(T\razn C)\eq g$ and consider the neighbourhood $G\eq G(\vphi,f,\eps/2)$ of~$\vphi$. It follows from $\vphi=\lim\rho$ that there is $i_1\in I$ such that $i_{\nu_i}\in G$ for every $i\geqslant i_1$. Since~$I$ is upward directed, there is $i_2\in I$ such that $i_2\geqslant i_0$ and $i_2\geqslant i_1$. Let $i\geqslant i_2$. Then
$|\vphi f|\leqslant|\vphi f-i_{\nu_i}f|+|i_{\nu_i}f|<\eps/2+i_{\nu_i}|f|\leqslant\eps/2+i_{\nu_i}g=
 \eps/2+\nu_i(T\razn C)<\eps$.
Thus, $\vphi$ is tight.

Now, according to the Bourbaki representation theorem \cite[Corollary~2 to Theorem~4~(3.6.3)]{ZakhRodi2018SFM2}, there is some measure $\mu_0\in\RM_b(T,\cG)_+$ such that $\vphi=i_{\mu_0}$. Check that $\mu_0=\lim\tau$. 

Take some neighbourhood
$H\eq G(\mu_0,\scol{f_k\in A(T)}{k\in n},\eps)$ of~$\mu_0$ and consider the corresponding neighbourhood 
$G\eq G(i_{\mu_0},\scol{f_k}{k\in n},\eps)$ of~$i_{\mu_0}$. Since $i_{\mu_0}=\lim\si$ there is $j_0\in J$ such that $i_{\mu_{\kap_j}}\in G$, i.\,e., $|i_{\mu_{\kap_j}}f_k-i_{\mu_0}f_k|<\eps$ for every $k\in n$ and every $j\geqslant j_0$. This means that $\mu_{\kap_j}\in M\cap H$ for $j\geqslant j_0$. By Proposition~1.1.1 from~\cite{Engelking1989} $\mu_0\in\cl M$. Since~$M$ is closed, this implies $\mu_0=\lim r$ and $\mu_0\in M$.

Thus, the net~$s$ has the subnet~$r$ converging to $\mu_0\in M$. Hence, by Theorem~2 from~\cite[ch.\,5]{Kelley1957} $M$ is weakly compact.
\end{proof}

Now we are going to show that the well known sufficiency theorems for a Tychonoff space (see~\cite[Th.\,1 (IX.5.5)]{Bourbaki(VI9)1969} and~\cite[8.6.7]{Bogachev2006}) directly follow from Theorem~\ref{theo-1-CwcTplus}. 

\begin{lemma}\label{lem-phifcont-CwcTplus}
Let $(T,\cG)$ be a Tychonoff space and $f\in C_b(T,\cG)$.
Then the function $\vphi_f:\RM_b(T,\cG)\to\Rbb$ such that $\vphi_f(\mu)\eq\int f\,d\mu$ for every $\mu\in\RM_b(T,\cG)$ is continuous on the topological space $(\RM_b,\cG_w(\RM_b,C_b))$.
\end{lemma}

The proof is completely similar to the proof of Lemma~\ref{lem-phifcont-SwcH} from section~\ref{sec-SwcH}.

\begin{proposition}\label{prop-psigsemicont-CwcTplus}
Let $(T,\cG)$ be a Tychonoff space and $g\in C_b(T,\cG)_+$.
Then the function $\psi_g:\RM_b(T,\cG)\to\Rbb$ such that $\psi_g(\mu)=\int g\,d|\mu|$ for every $\mu\in\RM_b$ is lower semicontinuous on the topological space $(\RM_b,\cG_w(\RM_b,C_b))$.
\end{proposition}

\begin{proof}
Consider the mapping $L:\RM_b\to C_b(T,\cG)^\pi$ such that $L(\mu)(f)\eq\int f\,d\mu$ for every $\mu\in\RM_b$ and $f\in C_b$.
By virtue of Corollary~1 to Theorem~2 (3.6.4) from~\cite{ZakhRodi2018SFM2} $L$ is an isomorphism of the given lattice-ordered linear spaces. Therefore $L(|\mu|)=|L(\mu)|$. 
By virtue of Corollary~1 to Theorem~5 (3.6.2) from~\cite{ZakhRodi2018SFM2} $C_b^\pi=C_b^\tu$.
Further, following the proof of Proposition~\ref{prop-psigsemicont-SwcH} and using Lemma~\ref{lem-phifcont-CwcTplus} instead of Lemma~\ref{lem-phifcont-SwcH} we get the necessary assertion.
\end{proof}

\begin{corollary}\label{cor-1-prop-psigsemicont-CwcTplus}
Let $(T,\cG)$ be a Tychonoff space, and $C$ be a compact set.
Then the function $\chi_1:\RM_b(T,\cG)\to\Rbb$ such that $\chi_1(\mu)=|\mu|(T\razn C)$ for every $\mu\in\RM_b$ is lower semicontinuous on the topological space $(\RM_b,\cG_w(\RM_b,C_b))$.
\end{corollary}

\begin{proof}
By Corollary~\ref{cor-1-lem-seqchiC-weaktop} 
$|\mu|(T\razn C)=\sup\{\int g\,d\mu\mid g\in C_b(T,\cG)_+ \wedge g\leqslant\chi(T\razn C)\}$.
In notations from the Proposition we have $\chi_1(\mu)=\sup\{\psi_g(\mu)\mid g\in C_b(T,\cG)_+ \wedge g\leqslant\chi(T\razn C)\}$ for every $\mu\in\RM_b$. This pointwise supremum means that $\chi_1=\sup\{\psi_g\mid g\in C_b \wedge g\leqslant\chi(T\razn C)\}$ in the lattice-ordered linear space $F(\RM_b)$. By the Proposition the function~$\psi_g$ is lower semicontinuous. Consequently, $\chi_1$ as the pointwise supremum of lower semicontinuous functions is lower semicontinuous as well.
\end{proof}

\begin{corollary}\label{cor-2-prop-psigsemicont-CwcTplus}
Let $(T,\cG)$ be a Tychonoff space.
Then the function $\chi_2:\RM_b\to\Rbb$ such that $\chi_2(\mu)=|\mu|(T)$ for every $\mu\in\RM_b(T,\cG)$ is lower semicontinuous on the topological space $(\RM_b(T,\cG),\cG_w(\RM_b(T,\cG),C_b(T,\cG)))$.
\end{corollary}

\begin{proof}
Apply the previous Corollary to $C\eq\vrn$.
\end{proof}

\begin{theorem}\label{theo-alzetabar-CwcTplus}
Let $(T,\cG)$ be a Tychonoff space, $M$ be a subset of the set $\RM_b(T,\cG)$, and $\cl M$ be the closure of~$M$ in the weak topology $\cG_w(\RM_b(T,\cG),C_b(T,\cG))$. Then
\begin{enumerate}
\item the following properties are equivalent:
	\begin{enumerate}
	\item[$(\al^{\pi}_{var})$] 
		for any $\eps>0$ there is a compact set~$C$ such that $|\mu|(T\razn C)<\eps$ for any $\mu\in M$;
	\item[$(\bar\al^{\pi}_{var})$] 
		for any $\eps>0$ there is a compact set~$C$ such that $|\nu|(T\razn C)<\eps$ for any $\nu\in\cl M$;
	\end{enumerate}
\item the following properties are equivalent:
	\begin{enumerate}
	\item[$(\beta_{var})$]
		$\sup\scol{|\mu| T}{\mu\in M}\in\Rbb_+$;
	\item[$(\bar\beta_{var})$]
		$\sup\scol{|\nu| T}{\nu\in\cl M}\in\Rbb_+$.
		\end{enumerate}
	\end{enumerate}
\end{theorem}

The proof is quite similar to the proof of Theorem~\ref{theo-alzetabar-SwcH} from section~\ref{sec-SwcH}. See also Proposition~11 (IX.5.5) in \cite{Bourbaki(VI9)1969}.

Prove the analogue of Lemma~\ref{lem-closposS-prsubs} for $c_b$-compactness.
 
\begin{lemma}\label{lem-closposCb-prsubs}
	Let $(T,\cG)$ be a Tychonoff space, $M\subset\RM_b(T,\cG)_+$, and $\cl M$ be the closure of~$M$ in the weak topology $\cG_w(\RM_b(T,\cG),C_b(T,\cG))$.
	Then $\cl M\subset\RM_b(T,\cG)_+$.
\end{lemma}

\begin{proof}
	Take some $\nu\in\cl M$, $B\in\cB$ such that $\eps\eq|\nu B|>0$, and $C\in\cC$ such that $C\subset B$ and $|\nu B-\nu C|<\eps$. Take also some net~$u$ from assertion~(i) of Lemma~\ref{lem-seqchiC-weaktop}. By Corollary~\ref{cor-3-lem-seqchiC-weaktop} to this Lemma $\nu C=\lim\net{\int f_k\,d\nu}{k\in K}$. Then for $\de\eq|\nu C|>0$ there is $k\in K$ such that 
	$|\nu C-\int f_l\,d\nu|<\de$ for every $l\geqslant k$. For $\beta\eq\int f_k\,d\nu$ and $\gm\eq|\beta|>0$ consider the neighbourhood $G\eq G(\nu,f_k,\gm)$. Since $\nu\in\cl M$, there is $\mu\in G\cap M$, i.\,e., $|\int f_k\,d\nu-\beta|<\gm$. This implies 
	$0\leqslant\int f_k\,d\mu<\beta+\gm=\beta+|\beta|$. If $\beta<0$, then $0<0$. It follows from this contradiction that $\beta\geqslant 0$.
	
	Using the inequality $|\nu C-\beta|<\de$ we get $0\leqslant\beta<\nu C+\de=\nu C+|\nu C|$. As above this implies $\nu C\geqslant 0$.
	Using the inequality $|\nu B-\nu C|<\eps$ we get $0\leqslant\nu C<\nu B+\eps=\nu B+|\nu B|$. As above this implies $\nu B\geqslant 0$.
	Thus, the measure~$\nu$ is positive.
\end{proof}

\begin{theorem}[the classical sufficiency theorem for positive measures]\label{theo-PBpos-CwcTplus}
Let $(T,\cG)$ be a Tychonoff space, $M\subset\RM_b(T,\cG)_+$ have properties $(\al^{\pi})$ and $(\beta)$, and $\cl M$ be the closure of~$M$ in the weak topology $\cG_w(\RM_b,C_b)$.
Then $\cl M$ is compact in the induced weak topology $\cG_w(\RM_b,C_b)|\cl M$.
\end{theorem}

\begin{proof}
First, note that by Lemma~\ref{lem-closposCb-prsubs} $\nu\geqslant\nul$ for every $\nu\in N\eq\cl M$ and prove that $N$ also possesses property $(\al^{\pi})$.

Let $\eps>0$. By condition there is $C\in\cC$ such that $\sup\scol{\mu(T\razn C)}{\mu\in M}\leqslant\eps/2$. 
Take $\nu\in N$ and $\de>0$. Consider the function $h\eq\chi(T\razn C)$. By Corollary~\ref{cor-1-lem-seqchiC-weaktop} to Lemma~\ref{lem-seqchiC-weaktop} there exists some net $u\eq\net{h_k}{k\in K}\uparrow$ such that $h=\plim u$ and $\net{i_{\nu}h_k}{k\in K}\uparrow\nu(T\razn C)=i_{\nu}h$.
Take $k\in K$ such that $0\leqslant i_{\nu}h-i_{\nu}h_{l}<\de$ for every $l\geqslant k$ and consider the neighbourhood 
$G\eq G(\nu,h_{k},\de)$. Take some $\mu\in M\cap G\ne\vrn$. Therefore we get 
$
\nu(T\razn C)=i_{\nu}h=i_{\nu}h-i_{\nu}h_{k}+i_{\nu}h_{k}<\de+i_{\mu}h_{k}+i_{\nu}h_{k}-i_{\mu}h_{k}
\leqslant\de+i_{\mu}h+|i_{\nu}h_{k}-i_{\mu}h_{k}|<\de+\mu(T\razn C)+\de\leqslant\eps/2+2\de.
$
Since~$\de$ is arbitrary, this implies $\nu(T\razn C)\leqslant\eps/2<\eps$.

According to Lemma~\ref{lem-betapr-prsubs}, $N$ possesses also property~$(\beta)$. Now the assertion follows immediately from Theorem~\ref{theo-1-CwcTplus} because property~$(\al^{\pi})$ is stronger than property~$(\al^{z})$.
\end{proof}

\begin{theorem}[the classical sufficiency theorem]\label{theo-PB-CwcTplus}
Let $(T,\cG)$ be a Tychonoff space, $\cl M$ be the closure of~$M\subset\RM_b(T,\cG)$ in the weak topology $\cG_w(\RM_b,C_b)$. Suppose that~$M$ has properties $(\al^{\pi}_{var})$ and~$(\beta_{var})$.
Then $\cl M$ is compact in the induced weak topology $\cG_w(\RM_b,C_b)|\cl M$.
\end{theorem}

\begin{proof}
By Theorem~\ref{theo-alzetabar-CwcTplus} the set $N\eq\cl M$ has properties $(\bar\al^{\pi}_{var})$ and~$(\bar\beta_{var})$.
Consider the subsets $L'\eq\{\nu^+\in\RM_b(T,\cG)_+\mid\nu\in N\}$ and $L''\eq\{-\nu^-\in\RM_b(T,\cG)_+\mid\nu\in N\}$ of positive bounded Radon measures, where $\nu^+\eq\nu\vee\nul$, $\nu^-\eq\nu\wedge\nul$. Since $\nu^+\leqslant|\nu|$, properties $(\bar\al^{\pi}_{var})$ and $(\bar\beta_{var})$ for~$N$ imply properties~$(\al^{\pi})$ and~$(\beta)$ for~$L'$. By Theorem~\ref{theo-PBpos-CwcTplus} $N'\eq\cl L'$ is compact in the induced weak topology $\cG_w(\RM_b,C_b)|N'$. The same is valid for~$N''\eq\cl L''$.

The remainder of the proof is quite similar to the proof of Theorem~\ref{theo-1prime-SwcH}.
\end{proof}

Probably, the last theorem was firstly published in~\cite[Th.\,1 (IX.5.5)]{Bourbaki(VI9)1969}.
According to~\cite[Preface]{Topsoe1970book} this theorem traces to P.-A.~Meyer and L.~Schwartz. 

\bibliographystyle{spmpsci} 
\bibliography{ZR-weak}

\end{document}